\documentclass[12pt,table]{amsart}

\usepackage{amsmath}
\usepackage[all]{xy}
\usepackage[monochrome]{xcolor}
\usepackage{amssymb,tikz}
\usepackage{slashbox,booktabs,amsmath}
\usepackage{graphicx}
\usepackage{adjustbox}
\usepackage{tikz}
\usepackage{tikz-cd}
\usetikzlibrary{cd}
\newcommand*\circled[1]{\tikz[baseline=(char.base)]{
		\node[shape=circle,draw,inner sep=0.5pt] (char) {#1};}}

\usepackage{rotating}

\def\BibTeX{{\rm B\kern-.05em{\sc i\kern-.025em b}\kern-.08em
    T\kern-.1667em\lower.7ex\hbox{E}\kern-.125emX}}

\newtheorem{thm}{Theorem}[section]
\newtheorem{cor}[thm]{Corollary}
\newtheorem{lem}[thm]{Lemma}
\newtheorem{prop}[thm]{Proposition}

\theoremstyle{definition}
\newtheorem{defn}{Definition}[section]
\theoremstyle{remark}
\newtheorem{rem}{Remark}[section]

\numberwithin{equation}{section}

\begin{document}

\title[Stability of the Heisenberg Product on Symmetric Functions] 
{Stability of the Heisenberg Product on Symmetric Functions}

%
\author{Li Ying}
\address{Li Ying, Department of Mathematics, Texas A\&M University, College Station, Texas, 77843, USA}
\email[]{98yingli@math.tamu.edu}
\urladdr{http://www.math.tamu.edu/~98yingli}
\thanks{}

\date{}


\begin{abstract}
The Heisenberg product is an associative product defined on symmetric functions which interpolates between the usual product and the Kronecker product. In 1938, Murnaghan discovered that the Kronecker product of two Schur functions stabilizes. We prove an analogous result for the Heisenberg product of Schur functions.
\end{abstract}


\maketitle

\section{Introduction}
\label{intro}
Aguiar, Ferrer Santos, and Moreira introduced a new product, the Heisenberg product, on symmetric functions (also on representations of symmetric group) in \cite{M-2015} and \cite{Moreira}. Unlike the ordinary product and the Kronecker product, the terms appearing in the Heisenberg product of two Schur functions have different degrees. The highest degree component is the usual product. When the Schur functions have the same degree, the lowest degree component of the Heisenberg product is their Kronecker product. 

In 1938, Murnaghan \cite{F-1938} found that the Kronecker product of two Schur functions stabilizes in the following sense. Given a partition $\lambda$ of $l$ and a large integer $n$, let $\lambda[n]$ be the partition of $n$ by prepending a part of size $n-l$ to $\lambda$. Given two partitions $\lambda$ and $\mu$, the coefficients appearing in the Schur expansion of the Kronecker product $s_{\lambda[n]}*s_{\mu[n]}$ do not depend upon $n$ when $n$ is large enough. The aim of this paper is to show that each degree component of the Heisenberg product also has this property.

The paper is organized as follows. In the second section, we first give the definitions of the induction product, the Kronecker product, and the Heisenberg product, and recall some important results. At the end of this section, we state the main result of this paper, which says that each degree component of the Heisenberg product has the similar stabilization property as the Kronecker product. Section \ref{example} offers an example of this stabilization. In the fourth section, we prove the main theorem. In the last section, we define the stable Heisenberg coefficients, and show how to recover the usual Heisenberg coefficients from the stable ones which generalizes an analogue formula for the Kronecker coefficients in \cite{E-2010}.

\section{Preliminaries}
\label{Preliminaries}
We begin by defining the induction product on (complex) representations of symmetric groups (we work with complex representations throughout the paper). For an introduction to representations of symmetric groups, see \cite{Sagan}. Let $V$ and $W$ be representations of $S_n$ and $S_m$ respectively. Observe that the tensor product $V\otimes W$ is a representation of $S_n\times S_m$, and $S_n\times S_m$ can be naturally embedded into $S_{n+m}$. The \textit{induction product} of $V$ and $W$ is the induced representation of $V\otimes W$ from $S_n\times S_m$ to $S_{n+m}$, written as $\text{Ind}_{S_n\times S_m}^{S_{n+m}}(V\otimes W)$. For any partition $\alpha$, let $V_\alpha$ denote the irreducible representation, known as the Specht module, of $S_{|\alpha|}$ indexed by $\alpha$. Let $\lambda$, $\mu$, and $\nu$ be partitions of $n$, $m$, and $n+m$\, respectively (written as $\lambda\vdash n$, $\mu\vdash m$, and $\nu \vdash m+n$). The \textit{Littlewood-Richardson coefficient} $c_{\lambda, \mu}^\nu$ is the multiplicity of $V_\nu$ in the decomposition of $\text{Ind}_{S_n\times S_m}^{S_{n+m}}(V_{\lambda}\otimes V_{\mu})$ into irreducible representations. That is,
\begin{equation}
\label{InductionProduct}
\text{Ind}_{S_n\times S_m}^{S_{n+m}}(V_{\lambda}\otimes V_{\mu})=\bigoplus\limits_{\nu\vdash n+m} c_{\lambda, \mu}^{\nu} V_\nu.
\end{equation}

Let $\langle\,\cdot\,,\,\cdot\,\rangle$ denote the natural inner product on the representations of the finite groups in which the irreducible representations form an orthonormal basis. Applying the Frobenius Reciprocity Theorem to (\ref{InductionProduct}), we have
\begin{align*}
c_{\lambda, \mu}^{\nu}&=\langle\;\text{Ind}_{S_n\times S_m}^{S_{n+m}}(V_{\lambda}\otimes V_{\mu})\,,\; V_\nu\;\rangle_{S_{n+m}}\\
&=\langle\; V_{\lambda}\otimes V_{\mu}\,,\;\text{Res}_{S_n\times S_m}^{S_{n+m}}V_\nu\;\rangle_{S_n\times S_m}.
\end{align*}
 So
$$\text{Res}_{S_n\times S_m}^{S_{n+m}} V_{\nu}=\bigoplus\limits_{\lambda\vdash n\,,\, \mu\vdash m} c_{\lambda, \mu}^{\nu} (V_\lambda\otimes V_\mu).$$

There is a one-to-one correspondence between the irreducible representations (up to isomorphism) and the Schur functions by the Frobenius characteristic map, which sends $V_\lambda$ to the Schur function $s_\lambda$. So we could also express the induction product in terms of symmetric functions. Under this bijection, the induction product corresponds to the ordinary product (denoted by $\cdot$) on symmetric functions, i.e.
$$s_\lambda\cdot s_\mu=\sum\limits_{\nu\vdash n+m} c_{\lambda, \mu}^\nu s_\nu.$$

The Littlewood-Richardson coefficient has been well-studied, it has the following nice combinatorial description:
\begin{prop}[Littlewood-Richardson rule, \cite{Mac} Chapter 1 Section 9]
	\label{LR}
	Let $\lambda$, $\mu$, and $\nu$ be partitions. Then $c_{\lambda, \mu}^\nu$ is equal to the number of semi-standard skew Young tableaux $T$ of shape $\nu-\lambda$ and weight $\mu$ whose reverse row reading word $w(T)$ is a lattice permutation.
\end{prop}
\vskip -1mm
\noindent (See \cite{Mac} for a more thorough introduction to the above notions.)\\

The Kronecker product can also be defined in terms of representations of symmetric groups. Let $V$ and $W$ be representations of $S_n$. While the tensor product $V\otimes W$ is a representation of $S_n\times S_n$, it can also be considered as a representation of $S_n$ (by viewing $S_n$ as a subgroup of $S_n\times S_n$ through the diagonal map). Write it as $\text{Res}_{S_n}^{S_n\times S_n} (V \otimes W)$. Let $\lambda$, $\mu$, and $\nu$ be partitions of $n$. The \textit{Kronecker coefficient} $g_{\lambda, \mu}^\nu$ is the multiplicity of $V_\nu$ in the decomposition of $\text{Res}_{S_n}^{S_n\times S_n} (V_\lambda \otimes V_\mu)$ into irreducibles. That is, 
$$\text{Res}_{S_n}^{S_n\times S_n} (V_\lambda \otimes V_\mu)=\bigoplus\limits_{\nu\vdash n} g_{\lambda, \mu}^{\nu} V_\nu.$$
Using the above formula, we can define the Kronecker product (denoted by $*$) for symmetric functions:
$$s_\lambda * s_\mu=\sum\limits_{\nu\vdash n} g_{\lambda, \mu}^\nu s_\nu.$$
We will switch between the languages of representation theory and symmetric functions. 

There is some interesting general work on representation stability by Church, Ellenberg, and Farb \cite{RepStab, FIforSn, HomolStab}. In this paper, we focus on the stability phenomenon of the Kronecker product discovered by Murnaghan \cite{F-1938}. 
We introduce some notations which will be used throughout the paper. Let $\alpha$ be a finite integer sequences. Define $\alpha^+$ to be the sequence obtained from $\alpha$ by adding $1$ to the first part $\alpha^+= (\alpha_1+1, \alpha_2, \alpha_3, \dotsc)$; similarly, set $\alpha^-= (\alpha_1-1, \alpha_2, \alpha_3, \dotsc)$. Let $\overline{\alpha}=(\alpha_2, \alpha_3, \dotsc)$ be the sequence obtained from $\alpha$ by removing the first part. 
Let $\beta$ be another finite integer sequence, we set $\alpha+\beta=(\alpha_1+\beta_1, \alpha_2+\beta_2, \dotsc)$ and $\alpha-\beta=(\alpha_1-\beta_1, \alpha_2-\beta_2, \dotsc)$. 

Given an eventually constant sequence $\left\{a_n\right\}_{n\geq 0}$ with stable value $L$, and $n_0$ the smallest integer, denoted by $\text{SStab}(\left\{a_n\right\})$, such that for all $n\geq n_0$, $a_{n}=L$. We say that this sequence stabilizes when $n\geq M$ if as long as $M\geq n_0$, and the stabilization begins at $n=n_0$. 
For a sequence of symmetric functions $\left\{F_n\right\}_{n\geq 0}$, where $F_n$ has the Schur expansion $F_n=\sum\limits_\alpha a_n^\alpha s_\alpha$ (we set $a_n^\alpha=0$ if $\alpha$ is not a partition). We say the sequence $\left\{F_n\right\}$ stabilizes if for any $\alpha$ (not necessarily a partition), the sequences $\left\{a_n^{\alpha+(n)} \right\}_{n\geq 0}$ is eventually constant, and there exist $N$, such that $\text{SStab}(\left\{a_n^{\alpha+(n)}\right\})\leq N$ for all $\alpha$. Let $n_1$ be the smallest $N$ having this property, and we denote it by $\text{FStab}(\left\{F_n\right\})$. 
From the definition, we have  $n_1=\text{FStab}(\left\{F_n\right\})=\max\limits_\alpha \{\text{SStab}(\left\{a_n^{\alpha+(n)}\right\})\}$. We say the sequence of symmetric functions $\left\{F_n\right\}$ stabilizes when $n\geq M$ as long as $M\geq n_1$, and the stabilization begins at $n=n_1$.

Given a partition $\lambda=(\lambda_1, \lambda_2, \dotsc)$ and a positive integer $n$, let $\lambda[n]$ be the sequence $(n-|\lambda|,\lambda_1, \lambda_2, \dotsc)$.  When $n \geq |\lambda| + \lambda_1$, $\lambda[n]$ is a partition of $n$. The stability of the Kronecker product means that for any partitions $\lambda$ and $\mu$, the sequence of symmetric functions $\left\{s_{\lambda[n]}*s_{\mu[n]}\right\}_{n\geq 0}$ stabilizes when $n$ is large enough. This phenomenon is best shown on an example. Let $\lambda=(2)$ and $\mu=(1,1)$, we compute the Kronecker product $s_{n-2,2}*s_{n-2,1,1}$ for $n\geq 4$:
\begin{alignat*}{10}
s_{2,2}*s_{2,1,1} & =s_{3,1} & & +&s_{2,1,1} & & & && & & \\
s_{3,2}*s_{3,1,1}&=s_{4,1}&+s_{3,2}& +&2s_{3,1,1}&& +&&s_{2,2,1}&+s_{2,1,1,1}&&\\
s_{4,2}*s_{4,1,1}&=s_{5,1}&+s_{4,2}& +&2s_{4,1,1}&+s_{3,3}& +&&2s_{3,2,1}&+s_{3,1,1,1}& &+s_{2,2,1,1}\\
s_{5,2}*s_{5,1,1}&=s_{6,1}&+s_{5,2}& +&2s_{5,1,1}&+s_{4,3}&+&&2s_{4,2,1}& +s_{4,1,1,1}&+s_{3,3,1}&+s_{3,2,1,1}\\
s_{6,2}*s_{6,1,1}&=s_{7,1}&+s_{6,2}& +&2s_{6,1,1}&+s_{5,3}&+&&2s_{5,2,1}&+s_{5,1,1,1}&+s_{4,3,1}&+s_{4,2,1,1}.
\end{alignat*}
Observe that the last two equations are only different in the first part of the indexing partitions. Indeed, for $n\geq 7$, we have
\begin{alignat*}{2}
s_{n-2,2}*s_{n-2,1,1}&=s_{n-1,1}+s_{n-2,2} +2s_{n-2,1,1}+s_{n-3,3}+2s_{n-3,2,1}\\
&+s_{n-3,1,1,1}+s_{n-4,3,1}+s_{n-4,2,1,1}.
\end{alignat*}
In this example, the stabilization of the sequence $\left\{g_{(n,2),(n-2,1,1)}^{(n-3,2,1)}\right\}_{n\geq 0}$ begins at $n=6$. The sequence of symmetric functions $\left\{s_{n-2,2}*s_{n-2,1,1}\right\}_{n\geq 0}$ stabilizes when $n\geq N$ as long as $N\geq 7$, and the stabilization begins at $n=7$. 

In the above example, one can also observe that, for fixed partition $\nu$, the sequence of coefficients of $s_{\nu[n]}$ in the expansion is weakly increasing as $n$ increases. This was shown by Brion \cite{Brion} and Manivel \cite{Manivel}:
\begin{prop}
	\label{KWeaklyIncrease}
	Let $\lambda$, $\mu$, and $\nu$ be partitions. The sequence $g_{\lambda[n],\mu[n]}^{\nu[n]}$ is weakly increasing. 	
\end{prop}
The sequence $\left\{g_{\lambda[n],\mu[n]}^{\nu[n]}\right\}$ is eventually constant according to the stability of the Kronecker coefficients. Write $\overline{g}_{\lambda, \mu}^\nu$ for the stable value of this sequence and call it a \textit{reduced Kronecker coefficient}. In our example, we see that $\overline{g}_{(2), (1,1)}^{(2,1)}=2$ and $\overline{g}_{(2), (1,1)}^{(1,1,1)}=1$.  Moreover, Murnaghan \cite{F-1938} claimed that $\overline{g}_{\lambda, \mu}^\nu$ vanishes unless
$$|\lambda|\leq |\mu|+|\nu|,\hskip 5mm  |\mu|\leq |\lambda|+|\nu|,\hskip 5mm  |\nu|\leq |\lambda|+|\mu|,$$
which are triangle inequalities. When $|\nu|= |\lambda|+|\mu|$, $\overline{g}_{\lambda, \mu}^\nu$ is equal to the Littlewood-Richardson coefficient $c_{\lambda, \mu}^\nu$ \cite{F-1938}.

Briand et al.~\cite{E-2010} determined when the Kronecker product stabilizes and provide another condition for the reduced Kronecker coefficient being nonzero.

\begin{prop}[\cite{E-2010} Theorem 1.2]
	\label{Kstability}
	Let $\lambda$ and $\mu$ be partitions. The sequence of symmetric functions $\left\{s_{\lambda[n]} * s_{\mu[n]}\right\}_{n\geq0}$ stabilizes, and the stabilization begins at $n=|\lambda|+|\mu|+\lambda_1+\mu_1$.
\end{prop}

\begin{prop}[\cite{E-2010} Theorem 3.2]
	\label{Knonzero}
	Let $\lambda$ and $\mu$ be partitions, then
	$$\max\{|\nu|+\nu_1\arrowvert\, \nu \;\text{parition},\, \overline{g}_{\lambda, \mu}^\nu> 0\}=|\lambda|+|\mu|+\lambda_1+\mu_1.$$
\end{prop}

\noindent Proposition \ref{Knonzero} will be used later in the proof of Theorem \ref{mainresult}.

Aguiar et al.~\cite{M-2015} and Moreira ~\cite{Moreira} introduced a new (nongraded) product which interpolates between the induction product and the Kronecker product.

\begin{defn}{(Heisenberg product)} Let $V$ and $W$ be representations of $S_n$ and $S_m$ respectively. Fix an integer $l\in[\text{max}\{m,n\}, m+n]$, and let $a=l-m$, $b=n+m-l$, and $c=l-n$. We have the (commutative) diagram of inclusions (solid arrows):	
\begin{equation}
\label{IncluDiag}
	\begin{tikzcd}
	S_p\times S_q\times S_q \times S_r \arrow[r, hook] & {S_{p+q} \times S_{q+r} =S_n\times S_m} \arrow[ddl, bend right=16, dashed, "\text{Res}"'] & \hskip -6mm \color{red}{V\otimes W}\\
	&&\\
	{S_p \times S_q \times S_r} \arrow[r, hook] \arrow[uu, hook, "id_{S_p}\times \Delta_{S_q}\times id_{S_r}"] \arrow[uur, hook] \arrow[r, bend left=16, dashed, "\text{Ind}"]& {S_{p+q+r} =S_l} & \hskip -25mm \color{red}{(V \# W)_l}
	\end{tikzcd}
\end{equation}
The Heisenberg product (denoted by $\#$) of $V$ and $W$ is
\begin{equation}
\label{HeisenbergProduct}
\begin{split}
V \# W= \bigoplus\limits_{l=\text{max}(n,m)}^{n+m} (V\# W)_l,
\end{split}
\end{equation}
where the degree $l$ component is defined using the dashed arrows in the diagram:
\begin{equation}
\label{DegreeComponent}
(V\# W)_l=\text{Ind}_{S_a\times S_b\times S_c}^{S_l}\text{Res}_{S_a\times S_b\times S_c}^{S_n\times S_m} (V\otimes W).
\end{equation}
\end{defn}
\noindent When $l= m+ n$, $(V\# W)_l=\text{Ind}_{S_n\times S_m}^{S_{n+m}} (V\otimes W)$, which is the induction product of representations; when $l=n=m$, $(V\# W)_l=\text{Res}_{S_l}^{S_l\times S_l} (V\otimes W)$, which is the Kronecker product of representations. The Heisenberg product connects the induction product and the Kronecker product. Remarkably, this product is associative \cite[Theorem 2.3, Theorem 2.4, Theorem 2.6]{M-2015}. The \textit{Heisenberg coefficient} $h_{\lambda, \mu}^\nu$ is the multiplicity of $V_\nu$ in $V_\lambda\# V_\mu$, i.e.
$$V_\lambda \# V_\mu = \bigoplus\limits_{l=\text{max}(n,m)}^{n+m} \bigoplus\limits_{\nu\vdash l} h_{\lambda, \mu}^\nu V_\nu,$$
and we set $h^{\nu}_{\lambda, \mu}=0$ if $\lambda$, $\mu$, or $\nu$ is not a partition. Similar to the Kronecker product, we can use the above formula to define the Heisenberg product (also denoted by $\#$) for symmetric functions:
$$s_\lambda \# s_\mu = \sum\limits_{l=\text{max}(n,m)}^{n+m} \sum\limits_{\nu\vdash l} h_{\lambda, \mu}^\nu s_\nu.$$
By the definition of the Heisenberg product (see diagram (\ref{IncluDiag})), when $b$ is much greater than $a$ and $c$, the right hand side of (\ref{DegreeComponent}) behaves like the Kronecker product. A natural question is whether we can develop a stability result for this degree component.
\begin{thm}
	\label{mainresult}
	Given nonnegative integers $r$ and $t$ and two partitions $\lambda$ and $\mu$, the sequence of symmetric functions of $\left\{(s_{\lambda[n]}\# s_{\mu[n-r]})_{n+t}\right\}_{n\geq 0}$ stabilizes, and the stabilization begins at $n=|\lambda| +|\mu|+\lambda_1+\mu_1+ 3t+ 2r$.
\end{thm}


\section{Example of the Stability of the Heisenberg Product}
\label{example}

We give an example of the stabilization of the Heisenberg product.

Let us take $\lambda=(1,1)$, $\mu=(1)$. We check the stability of the two lowest degree components of $s_{(1,1)[n]}\# s_{(1)[n-1]}$:
\vskip 3mm
$s_{1,1,1}\# s_{1,1}=(s_{2,1,1,1}+s_{1,1,1,1,1})+\color{red}{(s_{3,1}+s_{2,2}+2s_{2,1,1}+s_{1,1,1,1})}\color{black}+\color{blue}{(s_3+s_{2,1})},$
\vskip 3mm
$s_{2,1,1}\# s_{2,1}= (s_{4,2,1}+s_{4,1,1,1}+s_{3,3,1}+s_{3,2,2}+2s_{3,2,1,1}+s_{3,1,1,1,1}+s_{2,2,2,1}+s_{2,2,1,1,1})+\\
(s_{5,1}+3s_{4,2}+4s_{4,1,1}+2s_{3,3}+8s_{3,2,1}+6s_{3,1,1,1}+3s_{2,2,2}+6s_{2,2,1,1}+4s_{2,1,1,1,1}+s_{1,1,1,1,1,1})+\\
\color{red}{(s_5+5s_{4,1}+7s_{3,2}+9s_{3,1,1}+8s_{2,2,1}+7s_{2,1,1,1}+2s_{1,1,1,1,1})}\color{black}+\color{blue}{(s_4+3s_{3,1}+2s_{2,2}+3s_{2,1,1}+}\\
{s_{1,1,1,1})}.$
\vskip 4mm
\noindent The lowest degree component $(s_{(1,1)[n]}\# s_{(1)[n-1]})_n$ for $n\geq 5$:
\vskip 3mm
$(s_{3,1,1}\# s_{3,1})_5= \color{blue}s_{5}+3s_{4,1}+4s_{3,2}+4s_{3,1,1}+4s_{2,2,1}+3s_{2,1,1,1}+s_{1,1,1,1,1},$
\vskip 3mm
$(s_{4,1,1}\# s_{4,1})_6=\color{blue} s_6+ 3s_{5,1}+ 4s_{4,2}+ 4s_{4,1,1}+ 2s_{3,3}+ 5s_{3,2,1}+ 3s_{3,1,1,1}+ s_{2,2,2}+ 2s_{2,2,1,1}+ s_{2,1,1,1,1},$
\vskip 2.7mm
$(s_{5,1,1}\# s_{5,1})_7 =\color{blue}s_7+ 3s_{6,1}+4s_{5,2}+4s_{5,1,1}+2s_{4,3}+5s_{4,2,1}+3s_{4,1,1,1}+s_{3,3,1}+s_{3,2,2}+2s_{3,2,1,1}+s_{3,1,1,1,1},$
\vskip 2.7mm
$(s_{6,1,1}\# s_{6,1})_8 =\color{blue}s_8+ 3s_{7,1}+4s_{6,2}+4s_{6,1,1}+2s_{5,3}+5s_{5,2,1}+3s_{5,1,1,1}+s_{4,3,1}+s_{4,2,2}+2s_{4,2,1,1}+s_{4,1,1,1,1},$

\hskip 2.7mm $\vdots$\\
\vskip -4mm
\noindent To ease comparison, we create a table for this:

\begin{table}[ht]
	\resizebox{12cm}{!}{
		\begin{tabular}{c|ccccccccccc}
			$n$ & \multicolumn{11}{c}{\text{coefficients in} $(s_{n-2,1,1}\#s_{n-2,1})_n$}\\\hline
			$3$& \circled{$1$} & $1$ &       &       &       &      &       &       &       &      &    \\
			$4$& $\color{red}1$ & \circled{3}   & $2$   &  $3$  &       &       &  $1$  &       &       &       &    \\
			$5$& $\color{red}1$ & $\color{red}3$   & \circled{4}   &  \circled{4}  &  $\color{red}2$  &  $\color{red}5$  & \circled{3}   &       &       &       & $\color{red}1$\\
			$6$& $\color{red}1$ & $\color{red}3$   & $\color{red}4$   &  $\color{red}4$  &  \circled{2}  &  \circled{5}  & $\color{red}3$   &       &  \circled{1}  & \circled{2}   & \circled{1}\\
			$n\geq 7$& $\color{red}1$ & $\color{red}3$   & $\color{red}4$   &  $\color{red}4$  &  $\color{red}2$  &  $\color{red}5$  & $\color{red}3$   &  \circled{1}  &  $\color{red}1$  & $\color{red}2$   & $\color{red}1$\\
	\end{tabular}}
	\vskip 2mm
	\caption{Schur expansion of $(s_{n-2,1,1}\# s_{n-2,1})_n$ for $n\geq 3$. The circled numbers are where we estimate the corresponding sequence of Heisenberg coefficients will stabilize using Corollary \ref{HeisenbergColumnStable}.}
	\label{StableTable}
\end{table}
\vskip -7mm
\noindent where the coefficients are the coefficients in the expansion in the Schur basis, of, respectively (in this order):
\vskip 0.5mm
$\indent \hskip 10mm s_n, s_{(n-1,1)}, s_{(n-2,2)}, s_{(n-2,1,1)}, s_{(n-3,3)}, s_{(n-3,2,1)},s_{(n-3,1,1,1)},\\
\indent \hskip 65mm  s_{(n-4,3,1)}, s_{(n-4,2,2)}, s_{(n-4,2,1,1)}, s_{(n-4,1,1,1,1)}.$\\
\vskip -2.5mm
We can see that when $n\geq 7$, the Schur expansion of this degree component always has the same Heisenberg coefficients in the Schur expansion, and the only difference is the first part of the indexing partitions. The stabilization of the sequence of the lowest degree components of $s_{n-2,1,1}\# s_{n-2,1}$ happens at $n=7$ (using Theorem \ref{mainresult} with $r=1$ and $t=0$, the stabilization begins at $n= 2+1+1+1+2=7$). When $n\geq 7$, we have 
\vskip -3.5mm
\begin{eqnarray}\label{lowest}
\begin{split}
(s_{n-2,1,1}&\# s_{n-2,1})_{n}=s_{n}+3s_{n-1,1}+4s_{n-2,2}+4s_{n-2,1,1}+2s_{n-3,3}\\
&+5s_{n-3,2,1}+3s_{n-3,1,1,1}+s_{n-4,3,1}+s_{n-4,2,2}+2s_{n-4,2,1,1}\\
&+s_{n-4,1,1,1,1}.
\end{split}
\end{eqnarray}
\vskip -2.3mm
\noindent From Table \ref{StableTable}, we can also see that different columns (i.e. sequences $\left\{h_{(n-2,1,1), (n-2,1)}^{\nu[n]}\right\}$ for different $\nu$) stabilize at different steps, we give an estimate for this in the next section.

We also compute the second lowest degree component $(s_{(1,1)[n]}\# s_{(1)[n-1]})_{n+1}$ for $n\geq 5$, and create a table (see Table \ref{StableTable2} on the next page) for the result, where the coefficients are the coefficients in the expansion in the Schur basis, of, respectively (in this order):\\
\vskip -5mm
\indent $s_{n+1}, s_{(n,1)}, s_{(n-1,2)}, s_{(n-1,1,1)}, s_{(n-2,3)}, s_{(n-2,2,1)}, s_{(n-2,1,1,1)}, s_{(n-3,4)}, s_{(n-3,3,1)}, s_{(n-3,2,2)},\\
\indent s_{(n-3,2,1,1)}, s_{(n-3,1,1,1,1)}, s_{(n-4,5)}, s_{(n-4,4,1)}, s_{(n-4,3,2)}, s_{(n-4,3,1,1)}, s_{(n-4,2,2,1)}, s_{(n-4,2,1,1,1)},\\ \indent s_{(n-4,1,1,1,1,1)}, s_{(n-5,5,1)}, s_{(n-5,4,2)}, s_{(n-5,4,1,1)}, s_{(n-5,3,3)}, 4s_{(n-5,3,2,1)}, s_{(n-5,3,1,1,1)}, s_{(n-5,2,2,2)},\\
\indent s_{(n-5,2,2,1,1)}, s_{(n-5,2,1,1,1,1)}.$\\

\begin{sidewaystable}
	\vskip 168mm
	\centering
	\resizebox{21cm}{6.4cm}{
	\begin{tabular}{c|ccccccccccccccccccccccccccccc}
		{\LARGE $n$} & \multicolumn{29}{c}{{\LARGE coefficients in $(s_{n-2,1,1}\#s_{n-2,1})_{n+1}$}}\\
		&&&&&&&& &&&&&&& &&&&&&& &&&&&&&\\
		&&&&&&&& &&&&&&& &&&&&&& &&&&&&&\\\hline
		&&&&&&&& &&&&&&& &&&&&&& &&&&&&&\\
		&&&&&&&& &&&&&&& &&&&&&& &&&&&&&\\
		5 && 1 & 7 & 13 & 16 & 7 & 24 & 16 & & & 7 & 13 & 7 & & & & & & & 1 & & & & & & & & &\\
		&&&&&&&& &&&&&&& &&&&&&& &&&&&&&\\
		&&&&&&&& &&&&&&& &&&&&&& &&&&&&&\\
		6 && 1 & 7 & 15 & 17 & 13 & 33 & 19 & & 17 & 16 & 26 & 10 & & & & & 8 & 7 & 2 & & & & & & & & &\\
		&&&&&&&& &&&&&&& &&&&&&& &&&&&&&\\
		&&&&&&&& &&&&&&& &&&&&&& &&&&&&&\\
		7 && 1 & 7 & 15 & 17& 15 & 34 & 19 & 6 & 26 & 18 & 29 & 10 & & & 10 & 13 & 12 & 10 & 2& & & & & & & 1 & 2 & 1\\
		&&&&&&&& &&&&&&& &&&&&&& &&&&&&&\\
		&&&&&&&& &&&&&&& &&&&&&& &&&&&&&\\
		8 && 1 & 7 & 15 & 17& 15 & 34 & 19 & 8 & 27 & 18 & 29 & 10 & & 9 & 12 & 16 & 12 & 10 & 2 & & & & 1 & 4 & 3 & 1 & 2 & 1\\
		&&&&&&& &&&&&&& &&&&&&& &&&&&&&\\
		&&&&&&&& &&&&&&& &&&&&&& &&&&&&&\\
		9 && 1 & 7 & 15 & 17& 15 & 34 & 19 & 8 & 27 & 18 & 29 & 10 & 2 & 10 & 12 & 16 & 12 & 10 & 2 & & 2 & 3 & 1 & 4 & 3 & 1 & 2 & 1\\
		&&&&&&&& &&&&&&& &&&&&&& &&&&&&&\\
		&&&&&&&& &&&&&&& &&&&&&& &&&&&&&\\
		$n\geq 10$ &&  1 & 7 & 15 & 17& 15 & 34 & 19 & 8 & 27 & 18 & 29 & 10 & 2 & 10 & 12 & 16 & 12 & 10 & 2 & 1 & 2 & 3 & 1 & 4 & 3 & 1 & 2 & 1\\
		
	\end{tabular}
}
\vskip 10mm
	\caption{ Schur expansion of $(s_{n-2,1,1}\# s_{n-2,1})_{n+1}$ for $n\geq 5$.}
	\label{StableTable2}

\end{sidewaystable}

This computation shows that the sequence of the second lowest degree components of $s_{n-2,1,1}\# s_{n-2,1}$ stabilizes at $n=10$ (using Theorem \ref{mainresult} with $r=1$ and $t=1$, the stabilization begins at $n= 2+1+1+1+3+2=10$). When $n\geq 10$, we have
\begin{eqnarray}\label{second}
\begin{split}
(s_{n-2,1,1}&\# s_{n-2,1})_{n+1}=s_{n+1}+7s_{n,1}+15s_{n-1,2}+17s_{n-1,1,1}+15s_{n-2,3}\\
&+34s_{n-2,2,1}+19s_{n-2,1,1,1}+8s_{n-3,4}+27s_{n-3,3,1}+18s_{n-3,2,2}\\&
+29s_{n-3,2,1,1}+10s_{n-3,1,1,1,1}+2s_{n-4,5}+10s_{n-4,4,1}\\&
+12s_{n-4,3,2}+16s_{n-4,3,1,1}+12s_{n-4,2,2,1}+10s_{n-4,2,1,1,1}\\&
+2s_{n-4,1,1,1,1,1}+s_{n-5,5,1}+2s_{n-5,4,2}+3s_{n-5,4,1,1}+s_{n-5,3,3}\\&
+4s_{n-5,3,2,1}+3s_{n-5,3,1,1,1}+s_{n-5,2,2,2}+2s_{n-5,2,2,1,1}\\&
+s_{n-5,2,1,1,1,1}.
\end{split}
\end{eqnarray}

\section{Proof of Theorem \ref{mainresult}}
\label{Proof of Thm}

To prove Theorem \ref{mainresult}, we first prove a stability property of the Littlewood--Richardson coefficient.
\begin{lem}
\label{Lstability}
Let $\lambda$, $\mu$ and $\nu$ be partitions with $|\nu|=|\lambda|+|\mu|$,

(1) If $\nu_1-\nu_2\geq |\lambda|$, then $c_{\lambda, \mu}^\nu=c_{\lambda, \mu^+}^{\nu^+}$.

(2) If $\mu_1-\mu_2\geq |\lambda|$, then $c_{\lambda, \mu}^\nu=c_{\lambda, \mu^+}^{\nu^+}$.

\end{lem}

\begin{proof}
 By Proposition \ref{LR}, $c_{\alpha, \beta}^\gamma$ ($\alpha$, $\beta$, and $\gamma$ are partitions) counts the number of semi-standard skew tableaux of shape $\gamma/\beta$ and weight $\alpha$ whose row reading word is a lattice permutation. Let $T_{\alpha, \beta}^\gamma$ be the set of these tableaux. We show that $|T_{\lambda, \mu}^\nu|=|T_{\lambda, \mu^+}^{\nu^+}|$.

Note that $T_{\lambda, \mu}^\nu= \emptyset$ unless $\mu \subset \nu$, and $\mu \subset \nu$ if and only if $\mu^+ \subset \nu^+$, hence it is enough to consider the case $\mu \subset \nu$. The skew diagrams $\nu/\mu$ and ${\nu^+}/{\mu^+}$ differ only by a shift of the first row. Since $\nu_1-\nu_2\geq |\lambda|$, the first row (may be empty) of $\nu/\mu$ is disconnected from the rest of the skew diagram, and similarly for ${\nu^+}/{\mu^+}$. This gives us a natural bijection between $T_{\lambda, \mu}^\nu$ and $T_{\lambda, \mu^+}^{\nu^+}$. Hence 
$|T_{\lambda, \mu}^\nu|=|T_{\lambda, \mu^+}^{\nu^+}|$, and (1) is proved.

The proof of (2) is the same, as $\mu_1-\mu_2 \geq |\lambda|$ also implies that the first row of $\nu/\mu$ is disconnected from the rest of it.
\end{proof}

\begin{rem}
	\label{LMonoton}
When $\lambda$, $\mu$, and $\nu$ do not satisfy the conditions in Lemma \ref{Lstability}, the one unit shift of the first row may fail to be a bijection between $T_{\lambda, \mu}^\nu$ and $T_{\lambda, \mu^+}^{\nu^+}$. However, it is still a well-defined injection from $T_{\lambda, \mu}^\nu$ to $T_{\lambda, \mu^+}^{\nu^+}$, which means $c_{\lambda, \mu}^\nu \leq c_{\lambda, \mu^+}^{\nu^+}$. In other words, the sequence $\left\{c_{\lambda, \mu+(n)}^{\nu+(n)}\right\}$ is weakly increasing and is constant when $n$ is large.
\end{rem}

Theorem \ref{mainresult} states that $\text{FStab}(\left\{(s_{\lambda[n]}\# s_{\mu[n-r]})_{n+t}\right\}_n)=|\lambda| +|\mu|+\lambda_1+\mu_1+ 3t+ 2r$. We first show that $\text{FStab}(s_{\lambda[n]}\# s_{\mu[n-r]})_{n+t})\leq |\lambda| +|\mu|+\lambda_1+\mu_1+ 3t+ 2r$, i.e.
\begin{equation}
\label{Aguiar}
h_{\lambda[n], \mu[n-r]}^{\nu^-}= h_{\lambda[n+1], \mu[n-r+1]}^{\nu}
\end{equation}
for all $\nu\vdash n+t+1$ when $n\geq |\lambda| +|\mu|+\lambda_1+\mu_1+ 3t+ 2r$.

To prove \eqref{Aguiar}, we express the Heisenberg coefficient in terms of the Littlewood-Richardson coefficients and the Kronecker coefficients.

\begin{lem}
	\label{formulaheisenberg}
	For each $\nu\vdash l$,
	\begin{equation}
	\label{h=cg}
	h_{\lambda, \mu}^\nu=\sum\limits_{\shortstack{\scriptsize {$\alpha \vdash a, \rho \vdash c$, \scriptsize $\tau \vdash n$}\\ \scriptsize {$\beta, \eta, \delta \vdash b$}}} c_{\alpha, \beta}^\lambda\,\, c_{\eta, \rho}^\mu\,\, g_{\beta, \eta}^\delta\,\, c_{\alpha, \delta}^\tau\,\, c_{\tau, \rho}^\nu
	\end{equation}
	where $\max(n,m)\leq l\leq n+m$, $a=l-m$, $b=m+n-l$, and $c=l-n$.
\end{lem}

\begin{proof}
Consider the diagram \eqref{IncluDiag} we used to define the Heisenberg product. Given partitions $\lambda\vdash n$ and $\mu\vdash m$, $V_\lambda \otimes V_\mu$ is a representation of $S_n\times S_m (=S_{a+b}\times S_{b+c})$. We compute the Heisenberg product of $V_\lambda$ and $V_\mu$ in three steps.
\begin{equation}
\label{Step}
\raisebox{40pt}{
\xymatrix{
	{S_a\! \times\! S_b\! \times\! S_b\! \times\! S_c} \ar@/_1pc/[d]_{(2)} &&{S_{a+b}\! \times\! S_{b+c} =S_n\!\times\! S_m} \ar@/_1.5pc/[ll]_{(1)} \ar@{^{(}->}[ll];[]\\
	{S_a \times S_b \times S_c} \ar@{^{(}->}[u]
	\ar@{^{(}->}[urr] \ar@{^{(}->}[rr]_{(3)} \ar@{^{(}->}[dr]^{(3.1)}&&{S_{a+b+c} =S_l}\\
	&{S_{a+b}\times S_c}  \ar@{^{(}->}[ur]^{(3.2)}& \hskip -40mm =S_n\times S_c
	}}
\end{equation}

First, we restrict the representation from $S_n\times S_m$ to $S_a\times S_b\times S_b\times S_c$,

\begin{equation}
\tag*{(1)}
\text{Res}_{S_a\times S_b\times S_b\times S_c}^{S_n\times S_m} (V_\lambda \otimes V_\mu)= \bigoplus\limits_{\shortstack{\scriptsize {$\alpha \vdash a$}\\ \scriptsize {$\beta \vdash b$}}} \bigoplus\limits_{\shortstack{\scriptsize {$\eta \vdash b$}\\\scriptsize {$\rho \vdash c$}}} c_{\alpha, \beta}^\lambda\,\, c_{\eta, \rho}^\mu\, V_\alpha\otimes V_\beta \otimes V_\eta \otimes V_\rho.
\end{equation}

Second, pull back to $S_a\times S_b\times S_c$ along the diagonal map of $S_b$. For $\alpha \vdash a$, $\rho \vdash c$, and $\beta, \eta\vdash b$ we have,

\begin{align}
\tag*{(2)}
\text{Res}_{S_a\times S_b\times S_c}^{S_a\times S_b\times S_b\times S_c} (V_\alpha\otimes V_\beta \otimes V_\eta \otimes V_\rho)= \bigoplus\limits_{\delta\vdash b} g_{\beta, \eta}^\delta\, V_\alpha\otimes V_\delta\otimes V_\rho.
\end{align}

The final step is the induction from $S_a\times S_b\times S_c$ to $S_{a+b+c} (=S_l)$. Break this step into two substeps as in \eqref{Step}. Given $\alpha \vdash a$, $\delta\vdash b$, and $\rho\vdash c$, we have:
\begin{align*}
\tag*{(3)}
	\text{Ind}_{S_a\times S_b\times S_c}^{S_l} (V_\alpha\otimes V_\delta\otimes V_\rho)&= \text{Ind}_{S_n\times S_c}^{S_l}\text{Ind}_{S_a\times S_b\times S_c}^{S_n\times S_{c}} (V_\alpha\otimes V_\delta\otimes V_\rho)\\
	&=\bigoplus\limits_{\shortstack{\scriptsize {$\tau \vdash n$}\\ \scriptsize {$\nu \vdash l$}}} c_{\alpha, \delta}^\tau\,\, c_{\tau, \rho}^\nu\,\, V_\nu.
\end{align*}
Combining $(1)$, $(2)$, and $(3)$ together, gives

\begin{align*}
	(V_\lambda\otimes V_\mu)_l&=\text{Ind}_{S_a\times S_b\times S_c}^{S_{l}} \text{Res}_{S_a\times S_b\times S_c}^{S_a\times S_b\times S_b\times S_c} \text{Res}_{S_a\times S_b\times S_b\times S_c}^{S_{n}\times S_{m}} (V_\lambda\otimes V_\mu)\\
	&=\bigoplus\limits_{\shortstack{\scriptsize {$\alpha \vdash a, \rho \vdash c, \tau \vdash n$}\\ \scriptsize ${\beta, \eta, \delta \vdash b, \nu\vdash l}$}} c_{\alpha, \beta}^\lambda\,\, c_{\eta, \rho}^\mu\,\, g_{\beta, \eta}^\delta\,\, c_{\alpha, \delta}^\tau\,\, c_{\tau, \rho}^\nu\,\, V_\nu 
\end{align*}
So for $\nu\vdash l$,
\begin{equation*}
h_{\lambda, \mu}^\nu=\sum\limits_{\shortstack{\scriptsize {$\alpha \vdash a, \rho \vdash c$, \scriptsize $\tau \vdash n$}\\ \scriptsize {$\beta, \eta, \delta \vdash b$}}} c_{\alpha, \beta}^\lambda\,\, c_{\eta, \rho}^\mu\,\, g_{\beta, \eta}^\delta\,\, c_{\alpha, \delta}^\tau\,\, c_{\tau, \rho}^\nu,
\end{equation*}
as claimed.
\end{proof}

We set $c_{\lambda, \mu}^\nu=0$ when $\lambda$, $\mu$, or $\nu$ is not a partition. Then (\ref{h=cg}) holds for all sequences $\nu$ with sum $l$. Applying (\ref{h=cg}), to prove (\ref{Aguiar}), it is enough to show that, when  $n\geq |\lambda| +|\mu|+\lambda_1+\mu_1+ 3t+ 2r$,
\begin{equation}
\label{equationforstability1}
\begin{split}
\sum_{(\alpha, \beta, \eta, \rho, \delta, \tau) \in T} &c_{\alpha, \beta}^{\lambda [n]}\,\, c_{\eta, \rho}^{\mu[n-r]}\,\, g_{\beta, \eta}^\delta\,\, c_{\alpha, \delta}^\tau\,\, c_{\tau, \rho}^{\nu^-}= 
\\
&\sum_{(\alpha^*, \beta^*, \eta^*, \rho^*, \delta^*, \tau^*)\in T^*} c_{\alpha^*, \beta^*}^{\lambda [n+1]}\,\, c_{\eta^*, \rho^*}^{\mu[n+1-r]}\,\, g_{\beta^*, \eta^*}^{\delta^*}\,\, c_{\alpha^*, \delta^*}^{\tau^*}\,\, c_{\tau^*, \rho^*}^{\nu}
\end{split}
\end{equation}

\noindent for all $\nu \vdash n+t+1$, where 
\begin{eqnarray*}
T&=&\{(\alpha, \beta, \eta, \rho, \delta, \tau)\mid\alpha \vdash r+t, \rho\vdash t, \tau\vdash n, \beta, \eta, \delta\vdash n-r-t \};\\
T^*&=&\{(\alpha^*, \beta^*, \eta^*, \rho^*, \delta^*, \tau^*)\mid
\alpha^* \vdash r+t, \rho^*\vdash t,\tau^*\vdash n+1,\\ &&\hskip 70mm\beta^*,\eta^*, \delta^* \vdash n-r-t+1\}.
\end{eqnarray*}

Define $f: T\longmapsto \mathbb{Z}_{\geq 0}$ and $f^*: T^*\longmapsto \mathbb{Z}_{\geq 0}$ as follows:
$$f(\alpha, \beta, \eta, \rho, \delta, \tau)=c_{\alpha, \beta}^{\lambda [n]}\,\, c_{\eta, \rho}^{\mu[n-r]}\,\, g_{\beta, \eta}^\delta\,\, c_{\alpha, \delta}^\tau\,\, c_{\tau, \rho}^{\nu^-},$$
$$f^*(\alpha^*, \beta^*, \eta^*, \rho^*, \delta^*, \tau^*)=c_{\alpha^*, \beta^*}^{\lambda [n+1]}\,\, c_{\eta^*, \rho^*}^{\mu[n+1-r]}\,\, g_{\beta^*, \eta^*}^{\delta^*}\,\, c_{\alpha^*, \delta^*}^{\tau^*}\,\, c_{\tau^*, \rho^*}^{\nu}.$$
Then Equation (\ref{equationforstability1}) becomes:
\begin{equation}
\label{sumf=f*}
\sum_{u \in T} f(u) =\sum_{u^*\in T^*} f^*(u^*).
\end{equation}
Some terms in the sums of (\ref{sumf=f*}) vanish. Let us consider only the nonvanishing terms. 

\noindent Let $T_0 =T \smallsetminus f^{-1}(0)$ and $T^*_0= T^*\smallsetminus {f^*}^{-1}(0)$, then (\ref{sumf=f*}) is equivalent to
\begin{equation}
\label{f=f*sum}
\sum_{u \in T_0} f(u) =\sum_{u*\in T^*_0} f^*(u^*).
\end{equation}

\begin{lem}
	\label{proofmap}
	When $n\geq |\lambda| +|\mu|+\lambda_1+\mu_1+ 3t+ 2r$, the embedding $\varphi$ from $T$ to $T^*$:
	$$\varphi (\alpha, \beta, \eta, \rho, \delta, \tau)= (\alpha, \beta^+, \eta^+, \rho, \delta^+, \tau^+)$$ induces a map $\varphi|_{T_0}$ from $T_0$ to $T^*_0$. Moreover, $f|_{T_0}=f^*\circ \varphi|_{T_0}$.
\end{lem}

\begin{proof}
For all $u=(\alpha, \beta, \eta, \rho, \delta, \tau)\in T_0$, we show that $\beta$, $\eta$, $\delta$, and $\tau$ have large enough first parts so that we can apply Proposition \ref{Kstability} and Lemma \ref{Lstability} to the Kronecker coefficients and the Littlewood-Richardson coefficients appearing in the definition of $f$.

Since $n \geq |\lambda| +|\mu|+\lambda_1+\mu_1+ 3t+ 2r$, we have
$$\lambda[n]_1- \lambda[n]_2 = n-|\lambda|-\lambda_1 \geq |\mu|+\mu_1+ 3t+ 2r \geq r+t \hskip 2mm (=|\alpha|)$$
and
$$\mu[n-r]_1- \mu[n-r]_2 = n-r- |\mu|-\mu_1 \geq |\lambda|+\lambda_1 + 3t+ r \geq t \hskip 2mm (=|\rho|).$$
Using Lemma \ref{Lstability} (1), we get
$$c_{\alpha, \beta}^{\lambda[n]}= c_{\alpha, \beta^+}^{\lambda[n+1]} \quad \text{and} \hskip 5mm c_{\eta, \rho}^{\mu[n-r]}=c_{\eta^+, \rho}^{\mu[n+1-r]}.$$
As $\beta \subset \lambda[n]$,  $|\overline{\beta}|\leq |\lambda| < n-r-t$ and $(\overline{\beta})_1\leq \lambda_1$. Similarly, we have $|\overline{\eta}|\leq |\mu| < n-r-t$ and $(\overline{\eta})_1\leq \mu_1$. Since $\beta$ and $\eta$ are both partitions of $n-r-t$, they can be written as $\beta=\overline{\beta}[n-r-t]$ and $\eta=\overline{\eta}[n-r-t]$ respectively. They both have large first parts. More specifically, we have
$$n-r-t\geq |\lambda|+|\mu|+\lambda_1+\mu_1+2t+r \geq |\overline{\beta}|+|\overline{\eta}|+(\overline{\beta})_1+(\overline{\eta})_1.$$
By Proposition \ref{Kstability}, we have
$$g_{\beta, \eta}^\delta = g_{\beta^+, \eta^+}^{\delta^+}=\overline{g}_{\overline{\beta}, \overline{\eta}}^{\overline{\delta}}.$$

Followed from Proposition \ref{Knonzero}, $$|\overline{\delta}|+(\overline{\delta})_1\leq |\overline{\beta}|+|\overline{\eta}|+(\overline{\beta})_1+(\overline{\eta})_1$$
for otherwise $g_{\beta, \eta}^\delta=0$, which implies that $f(u)=0$, contradiction! Hence, 
$$|\delta|-\delta_1+\delta_2\leq |\lambda|+|\mu|+\lambda_1+\mu_1$$
which gives us
\begin{align*}
\delta_1-\delta_2 \geq n-r-t-|\lambda|-|\mu|-\lambda_1-\mu_1\geq 2t+r\geq |\alpha|
\end{align*}
Applying Lemma \ref{Lstability} (2), we get
$$c_{\alpha, \delta}^\tau= c_{\alpha, \delta^+}^{\tau^+}.$$
Since $c_{\alpha, \delta}^\tau \neq 0$, after Proposition \ref{LR}, we have
$$\tau_2\leq \delta_2+|\alpha| \quad \text{and} \quad \tau_1\geq \delta_1.$$
So
\begin{align*}
\tau_1-\tau_2 \geq \delta_1-(\delta_2+|\alpha|) \geq 2t+r-(r+t)= t = |\rho|.
\end{align*}
Hence, by Lemma \ref{Lstability} (2), we get
$$c_{\tau, \rho}^{\nu^-}= c_{\tau^+, \rho}^{\nu}.$$
So
\begin{equation}
\label{f=f*foreach}
f(\alpha, \beta, \eta, \rho, \delta, \tau)=f^*(\varphi (\alpha, \beta, \eta, \rho, \delta, \tau)) (\neq 0),
\end{equation}
\noindent which means $\varphi (T_0) \subset T^*_0$ and $f|_{T_0}=f^*\circ \varphi|_{T_0}$.
\end{proof}

To show that $\varphi$ is a bijection between $T_0$ and $T_0^*$, we need construct a reverse map.

\begin{lem}
	\label{ReverseMap}
		When $n\geq |\lambda| +|\mu|+\lambda_1+\mu_1+ 3t+ 2r$, the map $\phi: (\alpha, \beta, \eta, \rho, \delta, \tau)\longrightarrow (\alpha, \beta^-,\eta^-, \rho, \delta^-, \tau^-)$ is well-defined from $T_0^*$ to $T_0$. Moreover, $f^*|_{T_0^*}=f\circ \phi.$
\end{lem}
\begin{proof}
Take $u=(\alpha, \beta, \eta, \rho, \delta, \tau)\in T_0^*$, we first show that $\beta^-, \eta^-, \delta^-,$ and $\tau^-$ are partitions.
Since $f^*(u)\neq 0$, we get $c_{\alpha, \beta}^{\lambda[n+1]}\neq 0$. After Proposition \ref{LR}, we must have $\overline{\beta}\subset \lambda$ and $\lambda[n+1]-\beta_1\leq |\alpha|$. Hence,
$$\beta_1-\beta_2\geq (n+1-|\lambda|-|\alpha|)-\lambda_1\geq|\mu|+\mu_1+2t+r+1\geq 1.$$
So $\beta^-$ is a partition. Similarly, we can show $\eta^-$ is a partition. Using Proposition \ref{Kstability} and Proposition \ref{Knonzero} as we did in the proof of Lemma \ref{proofmap}, we see that $\delta^-$ is a partition for
$$\delta_1-\delta_2\geq 2t+r+1\geq 1.$$
As $c_{\alpha, \delta}^\tau\neq 0$, we have $\tau_1\geq \delta_1$ and $\tau_2\leq \delta_2+|\alpha|$. This shows that $\tau^-$ is a partition because
$$\tau_1-\tau_2\geq \delta_1-(\delta_2+|\alpha|)\geq t+1\geq 1.$$
Then by the same argument as in the proof of Lemma \ref{proofmap}, we can show that $f^*|_{T_0^*}=f\circ \phi$, which implies that $\phi(T_0^*)\subset T_0.$
\end{proof}

\begin{proof}[Proof of Theorem \ref{mainresult}]
Combining Lemma \ref{proofmap} and Lemma \ref{ReverseMap}, we know $\varphi$ is a bijection between $T_0$ and $T_0^*$. With this and (\ref{f=f*foreach}), we prove (\ref{f=f*sum}), and hence
$$\text{FStab}(s_{\lambda[n]}\# s_{\mu[n-r]})_{n+t})\leq |\lambda| +|\mu|+\lambda_1+\mu_1+ 3t+ 2r.$$

To prove that the stabilization begins at $|\lambda| +|\mu|+\lambda_1+\mu_1+ 3t+ 2r$, it is enough to show that there is $\nu\vdash n+t$ with $\nu_1=\nu_2$ (then $\nu^-$ is not a partition) such that $h_{\lambda[n], \mu[n-r]}^\nu\neq 0$ when $n = |\lambda| +|\mu|+\lambda_1+\mu_1+ 3t+ 2r$. We use the Formula (\ref{h=cg}) for $h_{\lambda[n], \mu[n-r]}^\nu\neq 0$ (replace $\lambda$ and $\nu$ by $\lambda[n]$ and $\mu[n-r]$ respectively, and set $l=n+t$), and take
\vskip -2mm
$$\alpha=(a)=(r+t), \rho=(c)=(t),$$ 
$$\beta=\lambda[n]-\alpha=(n-|\lambda|-r-t, \lambda_1, \lambda_2, \dots)=(|\mu|+\lambda_1+\mu_1+2t+r, \lambda_1, \dots),$$
$$\eta=\mu[n-r]-\rho=(n-r-|\mu|-t, \mu_1, \mu_2, \dots)=(|\lambda|+\lambda_1+\mu_1+2t+r, \mu_1, \dots),$$
$$\delta=(\overline{\beta}+\overline{\eta})[n-r-t]=(n-r-t-|\overline{\beta}|-|\overline{\eta}|,\beta_2+\eta_2, \beta_3+\eta_3, \dots)=(\lambda_1+\mu_1+2t+r,\lambda_1+\mu_1, \dots),$$
$$\tau=(\delta_1, \delta_2+|\alpha|, \delta_3, \dots)=(\lambda_1+\mu_1+2t+r, \lambda_1+\mu_1+r+t, \lambda_2+\mu_2, \dots),$$
$$\nu=(\tau_1, \tau_2+|\rho|, \dots)=(\lambda_1+\mu_1+2t+r, \lambda_1+\mu_1+2t+r, \lambda_2+\mu_2, \dots).$$\\
\noindent By the Pieri Rule, $1=c_{\alpha, \beta}^{\lambda[n]}=c_{\eta, \rho}^{\mu[n-d]}=c_{\alpha, \delta}^\tau=c_{\tau, \rho}^\nu$, as $\alpha$ and $\rho$ have only one part each. Since $|\overline{\delta}|=|\overline{\beta}|+|\overline{\eta}|$, we have $g_{\beta, \eta}^\delta=\overline{g}_{\overline{\beta}, \overline{\eta}}^{\overline{\delta}}=c_{\overline{\beta}, \overline{\eta}}^{\overline{\delta}}$ (note that $\overline{\delta}=\overline{\beta}+\overline{\eta}$) which is also nonzero according to Proposition \ref{LR}.

So $h_{\lambda[n],\mu[n-r]}^{\nu}\neq 0$ and $\nu_1=\nu_2=\lambda_1+\lambda_2+2t+r$, this proves that $n = |\lambda| +|\mu|+\lambda_1+\mu_1+ 3t+ 2r$ is where the stabilization begins.
\end{proof}

When $n<|\lambda| +|\mu|+\lambda_1+\mu_1+ 3t+ 2r$, following the same arguments as in the proof of Lemma \ref{proofmap} (except for using Proposition \ref{KWeaklyIncrease}  and Remark \ref{LMonoton} instead of Proposition \ref{Kstability} and Lemma \ref{Lstability}), we can show that the map $\varphi$ in Lemma \ref{proofmap} induces an injection from $T_0$ to $T^*_0$ with $f|_{T_0}\leq f^*\circ \varphi|_{T_0}$. This gives us the following corollary:

\begin{cor}
	Given three partitions $\lambda$, $\mu$, and $\nu$ and two nonnegative integers $r$ and $t$, the sequence $h_{\lambda[n],\mu[n-r]}^{\nu[n+t]}$ is weakly increasing.
\end{cor}


\section {Stable Heisenberg Coefficients}
\label{StableAC}

Given partitions $\lambda$, $\mu$, and $\nu$, Theorem \ref{mainresult} tells us that the sequence $\{h_{\lambda+(n), \mu+(n)}^{\nu+(n)}\}_{n=0}^\infty$ is eventually constant. We write $\overline{h}_{\lambda,\mu}^\nu$ for that constant value, and call it a \textit{stable Heisenberg coefficient}. The stable Heisenberg coefficient generalizes the reduced Kronecker coefficient. By the way we define the stable Heisenberg coefficient, we have 
$$\overline{h}_{\lambda,\mu}^\nu=\overline{h}_{\lambda+(n), \mu+(n)}^{\nu+(n)}, \hskip 3mm \text{for all nonnegative integers} \hskip 2mm n.$$
The reason we restrict $n$ to nonnegative integers is that $\lambda+(n)$, $\mu+(n)$, and $\nu+(n)$ need to be partitions. But we can drop this restriction and extend the definition by setting
$$\overline{h}_{\lambda-(n), \mu-(n)}^{\nu-(n)}=\overline{h}_{\lambda,\mu}^\nu, \hskip 3mm \text{for all nonnegative integers} \hskip 2mm n.$$
We call a finite integer sequence $\alpha=(\alpha_1, \alpha_2, \ldots, \alpha_k)$ an h-partition if $\alpha_2\geq \alpha_3\geq \cdots\geq \alpha_k>0$, then a stable Heisenberg coefficient, in the new definition, is indexed by three h-partitions. We have
$$\overline{h}_{\lambda,\mu}^\nu=\overline{h}_{\lambda+(n), \mu+(n)}^{\nu+(n)}, \hskip 3mm \text{for all integers} \hskip 2mm n.$$
where $\lambda$, $\mu$, and $\nu$ are h-partitions.

Murnaghan \cite{F-1938} pointed out that the reduced Kronecker coefficients determine the Kronecker product. Briand et al.~\cite[Theorem 1.1]{E-2010} gave an exact formula to recover the Kronecker coefficients from reduced ones, and Bowman et al.~\cite{PartitionAlgebra} interpreted this formula in terms of the representation theory of the partition algebra. Analogously, the stable Heisenberg coefficients also determine the Heisenberg product, even for small values of $n$. This can be proved using vertex operators on symmetric functions, and the idea of the proof is the same as the proof of the stability of the Kronecker product in \cite{JT}. 

Consider the lowest degree component of $s_{2,1,1}\# s_{2,1}$ as an example. Let $n=4$, then \eqref{lowest} gives us
\begin{equation}
\label{examplerecover}
\begin{split}
(s_{2,1,1}\# s_{2,1})_4&=s_{4}+3s_{3,1}+4s_{2,2}+4s_{2,1,1}+2s_{1,3}+5s_{1,2,1}\\&+3s_{1,1,1,1}+s_{0,3,1}+s_{0,2,2}+2s_{0,2,1,1}+s_{0,1,1,1,1}.
\end{split}
\end{equation}
By the Jacobi-Trudi determinant formula,
$$s_\lambda= \text{det} (h_{\lambda_j+i-j})_{i,j}$$
where $h_k$ is the complete homogeneous symmetric function, and we set $h_k=0$ when $k$ is negative and $h_0=1$. We no longer require $\lambda$ to be a partition, $\lambda$ can be any finite integer sequence. Then the Jacobi-Trudi determinant will give us $0$ or $\pm 1$ times some Schur function.
Applying Jacobi-Trudi determinant to the right hand side of \eqref{examplerecover}, we have
$$s_{1,3}=-s_{2,2}, \hskip 3mm s_{0,3,1}=-s_{2,1,1}, \hskip 3mm s_{0,2,1,1}=-s_{1,1,1,1}, \hskip 3mm \text{and}$$
$$s_{1,2,1}=s_{0,2,2}=s_{0,1,1,1,1}=0.$$
So \eqref{examplerecover} gives us
$$(s_{2,1,1}\# s_{2,1})_4=s_4+3s_{3,1}+2s_{2,2}+3s_{2,1,1}+s_{1,1,1,1},$$
which coincides with the result we had in Section \ref{example}. This example shows the process to recover the Heisenberg coefficients from the stable ones. The following theorem generalizes the formula in \cite[Theorem 1.1]{E-2010}, and recovers the Kronecker coefficient as a special case. 

\begin{thm}
	\label{RecoverTheorem}
	Let $\lambda, \mu$, and $\nu$ be partitions with $|\nu|\geq |\lambda|\geq |\mu|$, then
\begin{equation}
\label{RecoverFormula}
h_{\lambda,\mu}^\nu=\sum\limits_{i=1}^{4|\nu|-|\lambda|-|\mu|} (-1)^{i-1}\overline{h}_{\lambda,\mu}^{\nu^{\dag i}},
\end{equation}
where $\nu^{\dag i}= (\nu_i-i+1, \nu_1+1, \nu_2+1, \dotsc, \nu_{i-1}+1, \nu_{i+1}, \nu_{i+2}, \dotsc)$.
\end{thm}

Consider an example. From Section~\ref{example}, we know that $h_{(2,1,1), (2,1)}^{(2,2)}=2$.
On the other hand, using the Formula \eqref{RecoverFormula}, we have
\begin{equation}
\label{key}
\begin{split}
h_{(2,1,1), (2,1)}^{(2,2)}&=\overline{h}_{(2,1,1), (2,1)}^{(2,2)}-\overline{h}_{(2,1,1), (2,1)}^{(1,3)}+\overline{h}_{(2,1,1), (2,1)}^{(-2,3,3)}
\\&-\overline{h}_{(2,1,1), (2,1)}^{(-3,3,3,1)}+\cdots.
\end{split}
\end{equation}

From \eqref{lowest}, we have $$\overline{h}_{(2,1,1), (2,1)}^{(2,2)}=4, \hskip 4mm \overline{h}_{(2,1,1), (2,1)}^{(1,3)}=2,$$ 
and
$$\overline{h}_{(2,1,1), (2,1)}^{(2,2)^{\dag i}}=0, \hskip 2mm \text{when} \hskip 2mm i\geq 3.$$
So \eqref{key} gives us
$$h_{(2,1,1), (2,1)}^{(2,2)}=4-2=2.$$

\begin{proof}[Proof of Theorem \ref{RecoverTheorem}]
From Theorem \ref{mainresult}, we know that when $n\geq |\overline{\lambda}|+|\overline{\mu}|+\lambda_2+\mu_2+3(|\nu|-|\lambda|)+2(|\lambda|-|\mu|)-|\lambda|$, the Heisenberg coefficients of $(s_{\lambda+(n)}\# s_{\mu+(n)})_{n+|\nu|}$ stabilize, i.e.
\begin{equation*}
\begin{split}
(s_{\lambda+(n)}\# s_{\mu+(n)})_{n+|\nu|}&=\sum\limits_{\tau\vdash n+|\nu|} h_{\lambda+(n),\mu+(n)}^\tau s_\tau
\\&=\sum\limits_{\tau\vdash n+|\nu|} \overline{h}_{\lambda+(n),\mu+(n)}^\tau s_\tau
\\&=\sum\limits_{\tau\vdash n+|\nu|} \overline{h}_{\lambda,\mu}^{\tau-(n)} s_\tau.
\end{split}
\end{equation*}
So 
\begin{equation}
\label{stableformula}
\begin{split}
(s_\lambda\# s_\mu)_{|\nu|}=\sum\limits_{\tau\vdash n+|\nu|} \overline{h}_{\lambda,\mu}^{\tau-(n)} s_{\tau-(n)}.
\end{split}
\end{equation}

To get $h_{\lambda,\mu}^\nu$ from \eqref{stableformula}, we determine which $s_{\tau-(n)}$'s would give us $\pm s_\nu$. Suppose the length of $\tau$ is $l$. From the Jacobi-Trudi formula, we know that $s_{\tau-(n)} = \pm s_\nu$ if and only if the length of $\nu$ is at most $l$ and $(\tau_1-n, \tau_2, \tau_3, \dotsc, \tau_l) + (l-1, l-2, \dotsc, 0)$ is a permutation of $(\nu_1, \nu_2, \dotsc, \nu_l)+(l-1, l-2, \dotsc, 0)$. This happens when there is an $i$ ($1\leq i\leq l$) such that
$$\tau_1-n+(l-1)=\nu_i+l-i,$$ 
$$\tau_j+(l-j)=\nu_{j-1}+(l-j+1), \hskip 2mm j=2,3,4,\dotsc, i,$$
$$\tau_j+(l-j)=\nu_j+(l-j), \hskip 2mm j=i+1, i+2, \dotsc, l,$$
which is equivalent to
$$\tau-(n) = \nu^{\dag i},$$
and when this happens,
$$s_{\tau-(n)}= (-1)^{i-1}s_\nu.$$
So the coefficient of $s_\nu$ in $(s_\lambda\# s_\mu)_{|\nu|}$ is
\begin{equation}
h_{\lambda,\mu}^\nu=\sum\limits_{i=1}^{l} (-1)^{i-1}\overline{h}_{\lambda,\mu}^{\nu^{\dag i}}.
\end{equation}
Take $n=3|\nu|-|\lambda|-|\mu|\geq |\overline{\lambda}|+|\overline{\mu}|+\lambda_2+\mu_2+3(|\nu|-|\lambda|)+2(|\lambda|-|\mu|)-|\lambda|$, since $l\leq |\tau|= n+|\nu|$, (5.4) can be written as
$$h_{\lambda,\mu}^\nu=\sum\limits_{i=1}^{4|\nu|-|\lambda|-|\mu|} (-1)^{i-1}\overline{h}_{\lambda,\mu}^{\nu^{\dag i}}.$$
\end{proof}

Now we use Theorem 5.1 to estimate when $\left\{h_{\lambda[n], \mu[n-r]}^{\nu[n+t]}\right\}_n$ stabilizes for given partitions $\lambda$, $\mu$, and $\nu$ and nonnegative integers $r$ and $t$. 
\begin{cor}
	\label{HeisenbergColumnStable}
	The sequence of Heisenberg coefficients $\left\{h_{\lambda[n], \mu[n-r]}^{\nu[n+t]}\right\}_{n\geq 0}$ stabilizes when $n \geq \frac{1}{2}(|\lambda|+|\mu|+|\nu|+\lambda_1+\mu_1+\nu_1-1)+r+t$.
\end{cor}

\begin{proof}
The Formula (5.2) gives us

\begin{equation}
\label{specialstable}
h_{\lambda[n], \mu[n-r]}^{\nu[n+t]}=\sum\limits_{i=1}^{2n+4t+r} (-1)^{i-1}\overline{h}_{\lambda[n],\mu[n-r]}^{{\nu[n+t]}^{\dag i}},
\end{equation}
So $h_{\lambda[n], \mu[n-r]}^{\nu[n+t]}$ reaches the stable value when $\overline{h}_{\lambda[n],\mu[n-r]}^{{\nu[n+t]}^{\dag i}}=0$ for all $i\geq 2$. By Theorem \ref{mainresult}, $\left\{(s_{\lambda[n]}\# s_{\mu[n-r]})_{n+t}\right\}_{n\geq 0}$ stabilizes at $n=|\lambda|+|\mu|+\lambda_1+\mu_1+3t+2r =:m$, so
$$\overline{h}_{\lambda[n],\mu[n-r]}^{{\nu[n+t]}^{\dag i}}=\overline{h}_{\lambda[m],\mu[m-r]}^{{\nu[n+t]}^{\dag i}+(m-n)}=h_{\lambda[m], \mu[m-r]}^{{\nu[n+t]}^{\dag i}+(m-n)}.$$
Since for $i\geq 2$, we have
\begin{equation*}
\begin{split}
{\nu[n+t]}^{\dag i}+(m-n)=&(\nu_{i-1}-i+1, n+t-|\nu|+1, \nu_1+1, \dotsc,\nu_{i-2}+1, \nu_{i}, \nu_{i+1}, \dotsc)\\
&\hskip 30mm +(m-n)\\
=&(\nu_{i-1}-i+1+m-n, n+t-|\nu|+1, \nu_1+1, \dotsc, \nu_{i-2}+1,\nu_{i}, \dotsc).
\end{split}
\end{equation*}
When $n\geq \frac{1}{2}(|\lambda|+|\mu|+|\nu|+\lambda_1+\mu_1+\nu_1-1)+t+r$, we have
$$n+t-|\nu|+1>\nu_{i-1}-i+1+m-n, \hskip 3mm \text{for all} \hskip 3mm i\geq 2.$$
So $h_{\lambda[m], \mu[m-r]}^{{\nu[n+t]}^{\dag i}+(m-n)}=0$ for all $i\geq 2$, which proves the corollary.
\end{proof}

We go back to Table 1 and compute the lower bound for the stabilization of each column using Corollary 5.2. We circle the number corresponding to those lower bounds. We can see that, in this case, the lower bounds are the places where the stabilizations of the Heisenberg coefficients begin, except for $h_{(n-2,1,1), (n-2,1)}^{(n-3,3)}$, $h_{(n-2,1,1), (n-2,1)}^{(n-3,2,1)}$, and $h_{(n-2,1,1), (n-2,1)}^{(n-4,1,1,1,1)}$.

\bibliographystyle{amsplain}
\bibliography{AguiarReference}

\end{document}